\documentclass[a4paper,12pt]{article}
 \usepackage{latexsym} 
 \usepackage{amssymb,amsmath,amsthm} 
\usepackage{tikz}

\newtheorem{theorem}{Theorem}
\newtheorem{remark}{Remark}
 
 \newtheorem{proposition}{Proposition}

\newcommand{\R}{\mathbb{R}}

\renewcommand{\H}{\mathbb{H}}

\newcommand{\HD}{{\mathbb H}{(\mathbb D})}

\newcommand{\HU}{{\mathbb H}{(\mathbb U})}

\begin{document}
	
\title{Some extensions of quaternions and symmetries of simply connected space forms}

\author{Gerardo Arizmendi\thanks{gerardo.arizmendi@udlap.mx} \and Marco Antonio P\'{e}rez-de la Rosa\thanks{marco.perez@udlap.mx}}

\date{\small Department of Actuarial Sciences, Physics and Mathematics.\\ Universidad de las Am\'{e}ricas Puebla.\\
San Andr\'{e}s Cholula, Puebla. 72810. M\'{e}xico.}
\maketitle
	
\begin{abstract}
		It is known that the groups of Euclidean rotations in dimension 3 (isometries of $S^2$), general Lorentz transformations in dimension 4 (Hyperbolic isometries in dimension 3), and screw motions in dimension 3 can be represented by the groups of unit--norm elements in the algebras of real quaternions, biquaternions (complex quaternions) and dual quaternions, respectively.
		In this work, we present a unified framework that allows a wider scope on the subject and includes all the classical results related to the action in dimension 3 and 4 of unit--norm elements of the algebras described above and the algebra of split biquaternions as particular cases. We establish a decomposition of unit--norm elements in all cases and obtain as a byproduct a new decomposition of the group rotations in dimension 4.
\end{abstract}

	
\section{Introduction}\label{sec:introduction}

In several works (see, e.g., \cite{Majernik,FjelstadGal,Yaglom}) it is established that a two--component number system forming an algebraic ring can be written in the form $z=a+ub$, $a,b\in\mathbb{R}$, where $u$ is certain ``imaginary'' unit. The algebraic structure of a ring demands that the product of each two--component numbers $z_1=a_1+ub_1$ and $z_2=a_2+ub_2$,
\[z_1z_2=(a_1+ub_1)(a_2+ub_2)=a_1a_2+u(a_1b_2+b_1a_2)+u^2b_1b_2,\]
belongs to the ring as well. Therefore, $u^2=\beta+u\gamma$, $\beta,\gamma\in\mathbb{R}$. In \cite{KantorSolodownikov} is proven that any possible two--component number can be reduced to one of the three types enlisted below. The ``imaginary'' unit $u$ will be denoted by $i$, $\varepsilon$ or $j$, respectively, in each of the following cases:
\begin{itemize}
	\item[i)] The complex numbers $\mathbb{C}$, $z=a+ib$ with $i^2=-1$, if $\beta+\gamma^2/4<0$.
	\item[ii)] The dual numbers $\mathbb{D}$, $z=a+\varepsilon b$ with $\varepsilon^2=0$, if $\beta+\gamma^2/4=0$.
	\item[iii)] The hyperbolic numbers $\mathbb{S}$, $z=a+jb$ with $j^2=1$, if $\beta+\gamma^2/4>0$.
\end{itemize}

Just as the geometry of the Euclidean plane can be described with complex numbers, the geometry of the Minkowski plane and Galilean plane can be described with hyperbolic numbers and dual numbers, respectively (see \cite{Yaglom}).

Rotation matrices are used extensively for computations in geometry, kinematics,
physics, computer graphics, animations, and optimization problems
involving the estimation of rigid body transformations. According to \cite{Ozdemir}, in higher dimensional spaces, obtaining a rotation matrix using the inner product is impractical since each column and row of a rotation matrix must be a unit vector perpendicular to all other columns and rows, respectively. These constraints make it difficult to construct a rotation matrix using the inner product. Instead, in higher dimensional spaces, rotation matrices can be generated using various other methods such as unit quaternions, the Rodrigues formula, the Cayley formula, and the Householder transformation.

The Euler--Rodrigues parametrization of rotations is usually used to work out rotation products, nevertheless, if one wants to take a wider view of the subject we shall use the algebra of quaternions and its extensions suchs as biquaternions, dual quaternions and split biquaternions.

Despite the fact that in Hamilton's work the coefficients of quaternions could be either real or complex he did not mention the role of biquaternions in relativistic motions. In \cite{Silberstein} Silberstein showed how the Lorentz transformations could be represented by the action of unit biquaternions on the quaternions, while in \cite{Macfarlane} MacFarlane shown how to formulate Lorentz transformations
using $2 \times 2$ complex matrices (see also \cite{Buchheim,Edmonds,HestenesLasenby,KravchenkoShapiro,Synge,Ward}).

On the other side, in \cite{Clifford} Clifford sketched out a theory were the components of quaternions were dual numbers that allows on to represent three--dimensional rigid motions by means of the group of unit dual quaternions, just as the unit quaternions carry a representation of the group of three--dimensional Euclidean rotations.

Just as real quaternions give a computationally--efficient algorithm for dealing with rigid--body rotations in real time, dual quaternions give a com\-pu\-ta\-tiona\-lly--efficient algorithm for dealing with rigid-body motions that also include translations, such as the motion of joints in manipulators (see \cite{DooleyMcCarthy,KavanCollinsSullivanZara,LeeMesbahiv}).

For an account of the basic notions of kinematical states represented by the groups of unit-norm elements in the algebras of real, dual, complex, and complex dual quaternions we refer the reader to \cite{Delphenich}.

In \cite{Girard} is shown how various physical covariance groups such as $\text{SO}(3)$, the Lorentz group, the general theory of relativity group, the Clifford Algebra, $\text{SU}(2)$ and the conformal group can be related to the quaternion group. Furthermore, as applications, the quaternion calculus is introduced in crystallography, the kinematics of rigid body motion, the Thomas precession, the special theory of relativity, classical electromagnetism, the equation of motion of the general theory of relativity and Dirac's relativistic wave equation. Also, in \cite{HitzerNittaKuroe} and its references one can find a survey on the development of Clifford’s geometric algebra and some of its engineering applications, demostrating the benefit of developing problem solutions in a unified framework for algebra and geometry with the widest possible scope: from quantum computing and electromagnetism to satellite navigation, from neural computing to camera geometry, image processing, robotics and beyond (see also \cite{Sommer}).

The primary aim of the present work is to give a brief but complete account of the use of the extensions of quaternions in geometry. The general picture is the following. Starting from an 8 dimensional extension of quaternions, we restrict to the $6$ dimensional group given of elements such that $Q(q)=1$, where the quadratic form is defined by $Q(q)=q\overline q$. We define an action of this 6 dimensional group on some 4 dimensional subspace which we identify with $(\R^4,Q)$. Depending on the extension we obtain the Minkowski space, where the action are by Lorentz transformations, the Galilean four dimensional space, where the actions is by rotations and shears, and the Euclidean 4 dimensional space, where de action is by rotations. Moreover, by restricting to a subset we obtain the symply connected space forms $H^3$, $E^3$ and $S^3$ and the corresponding actions is by orientation preserving isometries. We work with biquaternions, dual-quaternions and split-biquaternions in a unified way, etending the known results to the case Galilean four dimensional space and rotations in $\mathbb R^4$ (orientation preserving isometries of $S^3$), giving a general picture of the construction. Our proofs are general in the sense that we threath the three cases together and we see the implications in each one, afterwards. 

The article is written as follows, Section \ref{sec:preliminaries} is devoted to some preliminary results. In Section \ref{sec:action}, we state and proof our first resuts in a general way, including a decompostion theorem (Theorem \ref{decomposition}), we also study the action of unit elements on some space isomorphic to $\R^4$. Next, on Section \ref{sec:distance}, we define a distance on a subset of $\R^4$ and prove that the action restricted to this space is by isometries. Finally, in Section \ref{sec:spaceforms}, we study every case in detail, explaining the geometry involved.

\section{Preliminaries of quaternionic analysis and its extensions}\label{sec:preliminaries}

Let $\mathbb{H}$, $\mathbb{H}(\mathbb{C})$, $\mathbb{H}(\mathbb{D})$ and $\mathbb{H}(\mathbb{S})$ denote the sets of real quaternions, biquaternions, dual quaternions and split biquaternions, respectively. If $q\in\mathbb{H}$, $q\in\mathbb{H}(\mathbb{C})$, $\mathbb{H}(\mathbb{D})$ or $\mathbb{H}(\mathbb{S})$ then 
\[q=q_0\mathbf{i}_0+q_1\mathbf{i}_1+q_2\mathbf{i}_2+q_3\mathbf{i}_3,\]
where $\mathbf{i}_0=1$, $\mathbf{i}_1^2=\mathbf{i}_2^2=\mathbf{i}_3^2=-1$ and satisfy the following multiplication rules \[\mathbf{i}_1\cdot \mathbf{i}_2=-\mathbf{i}_2\cdot \mathbf{i}_1=\mathbf{i}_3;\quad \mathbf{i}_2\cdot \mathbf{i}_3=-\mathbf{i}_3\cdot \mathbf{i}_2=\mathbf{i}_1;\quad \mathbf{i}_3\cdot \mathbf{i}_1=-\mathbf{i}_1\cdot \mathbf{i}_3=\mathbf{i}_2.\]

The coefficients $\{q_k\}\subset\mathbb{R}$ if $a\in\mathbb{H}$,
$\{q_k\}\subset\mathbb{C}$ if $a\in\mathbb{H}(\mathbb{C})$, $\{q_k\}\subset\mathbb{D}$ if $a\in\mathbb{H}(\mathbb{D})$ and $\{q_k\}\subset\mathbb{S}$ if $a\in\mathbb{H}(\mathbb{S})$.

Throughout the text we will denote any of the three units $i$, $\varepsilon$ or $j$ by $u$ and any of the rings $\mathbb{C}$, $\mathbb{D}$ or $\mathbb{S}$ by $\mathbb{U}$. By definition, $u$ commutes with all the quaternionic imaginary units $\mathbf{i}_1,\mathbf{i}_2,\mathbf{i}_3$.

Denoting for $q\in\mathbb{H}$ or $\mathbb{H}(\mathbb{U})$, $\vec{q}:=q_1\mathbf{i}_1+q_2\mathbf{i}_2+q_3\mathbf{i}_3$, we can write 
\[q=q_0+\vec{q},\]
where $\mathrm{Sc}(q):=q_0$ will be called the scalar part and $\mathrm{Vec}(q):=\vec{q}$ the vector part of the quaternion.

Also, if $q\in\mathbb{H}(\mathbb{U})$ and $q_k=\alpha_k + u \beta_k $  with  $\{ \alpha_k , \beta_k \}  \subset \mathbb{R} $, then:

\begin{align*}
q  & =    \sum_{k=0}^3 q_k \mathbf{i}_k = \sum_{k=0}^3 \left( \alpha_k + u \beta_k  \right) \mathbf{i}_k \\
& =  \sum_{k=0}^3 \left( \alpha_k \mathbf{i}_k + u \beta_k \mathbf{i}_k  \right) \\  & =: \alpha + u \beta =  \left(\alpha_0 + \vec{\alpha} \right) + u \left( \beta_0 + \vec{\beta} \right)\\  
& =  \left(  \alpha_0 + u \beta_0  \right)  +  \left( \vec{\alpha}  +  u  \vec{\beta} \right).
\end{align*} 

For any $p,q\in\mathbb{H}$ or $\mathbb{H}(\mathbb{U})$:
\[
p\,q := p_0\,q_0 - \langle\vec{p},\vec{q}\rangle + p_0\,\vec{q} + q_0\,\vec{p} + [\vec{p},\vec{q}],
\]
where
\[
\langle\vec{p},\vec{q}\rangle := \sum_{k=1}^3 p_k\,q_k,\,\,\;\;\; [\vec{p},\vec{q}]:=
\left |
\begin{array}{rrr}
\mathbf{i}_1 & \mathbf{i}_2 & \mathbf{i}_3\\
p_1& p_2 & p_3\\
q_1& q_2 & q_3
\end{array}
\right |.
\]

In particular, if $p_0=q_0=0$ then $p\,q :=  - \langle\vec{p},\vec{q}\rangle + [\vec{p},\vec{q}]$.

Let $q \in \mathbb{H}(\mathbb{U})$. We define the following involutions:
		\begin{align} 
		q^{\ast} &:= \alpha - u\beta = \sum_{k=0}^3{\left(\alpha_k - u \beta_k \right) \mathbf{i}_k};\label{Complex_Conjugation}\\
		\overline{q} &:= \overline{\alpha} + u\overline{\beta} = \sum_{k=0}^3{\left(\alpha_k + u\beta_k \right)\overline{\mathbf{i}_k}};\label{Quaternionic_Conjugation}\\
		\overline{q^{\ast}} &:= \overline{\alpha} - u\overline{\beta} = \sum_{k=0}^3{\left( \alpha_k - u \beta_k\right) \overline{\mathbf{i}_k}}.\label{Bi-Quaternionic_Conjugation}
		\end{align}

Observe that $\overline{q^{\ast}} = \overline{(q^{\ast})} = (\overline{q})^\ast$.

The involution (\ref{Complex_Conjugation}) satisfies the following properties:
\begin{enumerate}
	\item[a)] $(p\pm q)^{\ast} = p^{\ast}\pm q^{\ast}$,
	\item[b)] $(pq)^{\ast} = p^{\ast}q^{\ast}$,
	\item[c)] $\left(q^{\ast}\right)^\ast = q$,
	\item[d)] $\left(q^{\ast}q \right)^\ast = q q^{\ast}$.
	\item[e)] $\alpha=\frac{1}{2}(q+q^\ast)$.
	\item[f)] $\beta=\frac{1}{2u}(q-q^\ast)$.
\end{enumerate}

A direct observation shows that
\begin{equation*}
	q^{\ast}q = \alpha^2  -u^2\beta^2 + u\left(\alpha\beta - \beta \alpha \right)\in\mathbb{H}(\mathbb{U}),
\end{equation*}
and
\begin{equation*}
	q q^{\ast} = \alpha^2 - u^2\beta^2 - u\left(\alpha\beta - \beta\alpha \right)\in\mathbb{H}(\mathbb{U}).
\end{equation*}

The involution (\ref{Quaternionic_Conjugation}) satisfies the following properties:
\begin{enumerate}
	\item[a)] $\overline{p \pm q} = \overline{p} \pm \overline{q}$,
	\item[b)] $\overline{p q} = \overline{q}\, \overline{p}$,
	\item[c)] $\overline{\overline{q}} = q$,
	\item[d)] $q \overline{q} = \overline{q}q$,
	\item[e)] ${\rm Sc}(q) = \frac{1}{2}\left(q+\overline{q}\right)$,
	\item[f)] ${\rm Vec}(q) = \frac{1}{2}\left(q-\overline{q}\right)$,
	\item[g)] ${\rm Sc}\left(\overline{q}\right) = {\rm Sc}(q)$.
\end{enumerate}

A straightforward computation shows that
\begin{equation} \label{Quaternionic_Conjugation_Equality}
q  \overline{q} =\overline{q}q=  |\alpha|^2+u^2|\beta|^2+2u\left\langle\alpha,\beta\right\rangle \in\mathbb{U} .
\end{equation}

The involution (\ref{Bi-Quaternionic_Conjugation}) satisfies the following properties:
\begin{enumerate}
	\item[a)] $\overline{(p+q)^{\ast}} = \overline{p^{\ast}} + \overline{q^{\ast}}$,
	\item[b)] $\overline{(pq)^{\ast}} = \overline{q^{\ast}}\, \overline{p^{\ast}}$,
	\item[c)] $\overline{(\overline{q^{\ast}})^{\ast}} = q$.
\end{enumerate}

Observe that
	\begin{equation*}
		\overline{q^{\ast}}q = |\alpha|^2 -u^2 |\beta|^2 + 2u {\rm Vec}\left( \overline{\alpha} \beta \right)\in\mathbb{H}(\mathbb{U}),
	\end{equation*}
	and
	\begin{equation*}
		q \overline{q^{\ast}}= |\alpha|^2 -u^2 |\beta|^2 - 2u {\rm Vec}\left( \alpha \overline{\beta} \right)\in\mathbb{H}(\mathbb{U}).
	\end{equation*}

\subsection{Bilinear form}

Consider the following symmetric bilinear form for two elements $p,q\in\mathbb{H}(\mathbb{U})$:
\begin{equation}
\left\langle p,q\right\rangle_{\mathbb{H}}:=\mathrm{Sc}(p\overline{q}),
\end{equation}
that satisfies the following properties:
\begin{enumerate}
	\item[a)] $\left\langle p,q\right\rangle_{\mathbb{H}}=\mathrm{Sc}(p\overline{q})=\mathrm{Sc}\left(\overline{p\overline{q}}\right)=\mathrm{Sc}(q\overline{p})=\left\langle q,p\right\rangle_{\mathbb{H}}$,
	\item[b)] $\left\langle p,q+r\right\rangle_{\mathbb{H}}=\mathrm{Sc}\left(p\overline{(q+r)}\right)=\mathrm{Sc}\left(p\overline{q}+p\overline{r}\right)=\mathrm{Sc}\left(p\overline{q}\right)+\mathrm{Sc}\left(p\overline{r}\right)=\left\langle p,q\right\rangle_{\mathbb{H}}+\left\langle p,r\right\rangle_{\mathbb{H}}$,
	\item[c)] $\lambda\left\langle p,q\right\rangle_{\mathbb{H}}=\lambda\mathrm{Sc}\left(p\overline{q}\right)=\mathrm{Sc}\left((\lambda p)\overline{q}\right)=\left\langle \lambda p,q\right\rangle_{\mathbb{H}}=\mathrm{Sc}\left( p(\overline{\lambda q})\right)=\left\langle  p,\lambda q\right\rangle_{\mathbb{H}}$, $\lambda$ in $\mathbb{U}$.
	\item[d)] $\left\langle q,q\right\rangle_{\mathbb{H}}=\mathrm{Sc}(q\overline{q})=|\alpha|^2+u^2|\beta|^2+2u\left\langle\alpha,\beta\right\rangle\in\mathbb{U}$.
\end{enumerate}

The previous symmetric bilinear form defines the following quadratic form:
\begin{equation}
Q(q):=\left\langle q,q\right\rangle_{\mathbb{H}}=\mathrm{Sc}(q\overline{q})=q\overline{q}.
\end{equation}

Observe that $Q(q)=Q(\overline{q})$ and
\[Q(pq)=pq\overline{pq}=pq\overline{q}\overline{p}=q\overline{q}p\overline{p}=Q(p)Q(q),\]
using the fact that $q\overline{q}\in\mathbb{U}$ and commutes with all extensions of quaternions.

Define the angle between $p,q\in \HD$ as follows:
\begin{equation}
\cos(z):=\frac{\left\langle p,q\right\rangle_{\mathbb{H}}}{\sqrt{Q(p)}\sqrt{Q(q)}}.
\end{equation}

\subsection{Polar form}

For $z\in \mathbb{U}$ we can write down the power series for the trigonometric functions:

\begin{align*}
\sin z =& \sum^{\infty}_{n=0} \frac{(-1)^n}{(2n+1)!} z^{2n+1},\\
\cos z =& \sum^{\infty}_{n=0} \frac{(-1)^n}{(2n)!} z^{2n}.
\end{align*}

Observe that this definition implies that

\begin{align*}
\cos(z)&=\cos(a+ub)=\cos(a)\cos(ub)-\sin(a)\sin(ub),\\
\sin(z)&=\sin(a+ub)=\sin(a)\cos(ub)+\cos(a)\sin(ub).
\end{align*}

\begin{remark}
	Since $u^2\in \mathbb R$ then it follows that $\sin(ub)\in u\mathbb R$ and $\cos(ub)\in{\mathbb R}$
\end{remark}

\bigskip

Any $q\in\mathbb{H}(\mathbb{U})$ such that $Q(q)\neq0$ can be written in the following form:
\begin{equation}\label{polar}
q=\sqrt{Q(q)}\left(\cos(z)+\hat{q}\sin(z)\right),\quad Q(\hat{q})=1,\quad \hat{q}^2=-1,
\end{equation}
where $\cos(z)=\frac{q_0}{\sqrt{Q(q)}}$, $\sin(z)=\frac{\sqrt{Q(\vec{q})}}{\sqrt{Q(q)}}$ and $\hat{q}=\frac{\vec{q}}{\sqrt{Q(\vec{q})}}$.

Observe that for $q\in\mathbb{H}(\mathbb{C})$, $Q(q)=0$ if and only if $\left|\alpha\right|=\left|\beta\right|$ and $\left\langle \alpha,\beta\right\rangle=0$. For $q\in\mathbb{H}(\mathbb{D})$, $Q(q)=0$ if and only if $\left|\alpha\right|=0$, i.e. $q=\epsilon\beta$. For $q\in\mathbb{H}(\mathbb{S})$, $Q(q)=0$ if and only if $\left|\alpha\right|+\left|\beta\right|=0$, i.e. $q\equiv0$.

\section{The action of unit--norm elements on $\mathbb R^4$}\label{sec:action}

We now present an important result that will be very useful later.

\begin{theorem}\label{decomposition} Let $q\in \HU$ such that $Q(q)=1$, then there exist $q_r$, and $q_b$ such that $q=q_rq_b$ with $Q(q_r)=Q(q_b)=1$, $q_r^*=q_r$ and $\overline{q_b}=q^*_b$.
\end{theorem}
\begin{proof} Let $q=\alpha'+u\beta'$, with $\alpha',\beta'\in \H$ then since $Q(q)=1$ this means that  and $\left\langle \alpha',\beta' \right\rangle=0$ and $|\alpha'|^2+u^2|\beta'|^2=1$. The last equation implies that we can write $|\alpha'|=\cos(z)$ and  $u|\beta'|=\sin(z)$ for some $z\in\mathbb {U}$, moreover, since $\alpha'\in \H$ then $z=u\theta$ for some $\theta \in \mathbb{R}$.

Hence, we can write $q=\cos(u\theta)\alpha+ \sin(u\theta)\beta$ with $\theta \in \mathbb{R}$, $\alpha,\beta\in \H$, $Q(\alpha)=Q(\beta)=1$  and  $\left\langle \alpha,\beta \right\rangle=0$. Now let $q_r=\alpha$ and and $q_b=\cos(u\theta)+ \sin(u\theta)\overline{\alpha}\beta$, then since  $\left\langle \alpha,\beta \right\rangle=0$, it follows that $\overline \alpha\beta$ is a pure real vector. This  proves the claim, since $Q(\overline \alpha\beta)=\overline \alpha\beta\overline{\overline \alpha\beta}=\overline \alpha\beta\overline \beta\alpha=1$.
\end{proof}

\begin{remark} In the decomposition $q=q_rq_b$ above, the elements such that  $q_r^*=q_r$ form a group. In general, this is not the case for elements such that $\overline{q_b}=q^*_b$. For example, for the case $u=i$, as we will see in section \ref{sec:spaceforms}, $q_b$ represents a boost and the composition of two boosts involves a spatial rotation and a boost in the decomposition, the rotation is known as ``Thomas rotation" and the effect is called Thomas precession.

\end{remark}

We identify the subset of $\HU$ given by $\{v=v_0+u\vec{v} \mbox{ }|\mbox{ } v_0\in \mathbb{R}, \mbox{ }\vec{v}\in Vec(\H)\}$ with $\mathbb{R}^4$ in the most natural way (these elements are called minquats as a short name for minkowskian quaternions in \cite{Synge} for the case $u=i$).  Observe that this space is the same as $\{v\in \HU|v^*=\overline{v}\}$, moreover from Formula (\ref{polar}), if $Q(v)\neq 0$, we can write $v=\|v\|(\cos(u\phi)+\hat{v}\sin(u\phi))$, where $\|v\|$ is the standard norm of $\R^4$ and $\hat{v}^2=-1$. 

 Let $q\in\HU$ such that $Q(q)=1$ and consider 
 \[T_q(v):=qv\overline{q}^*.\]
 Obvserve that 
 \[T_{pq}(v)=pqv\overline{pq}^*=pqv\overline{q}^*\overline{p}^*=T_p(T_q(v))\]

so that 
\begin{equation}\label{morphism}
T_{pq}=T_p\circ T_q 
\end{equation}
hence the application $q\mapsto T_q$ is a group morphism.

 We will prove now the next

\begin{proposition} Let $q\in \HU$ such that $Q(q)=1$ and let $v\in\mathbb{R}^4$ then $T_q(v)\in \mathbb {R}^4$.  
\end{proposition}
\begin{proof}
 Since $Q(q)=1$, by Theorem \ref{decomposition} we can write $q=q_rq_b$ where  $q_r\in \H$, $q_b\in \R^4$ and $Q(q_r)=Q(q_b)=1$. Then, by equation (\ref{morphism}), $T_q=T_{q_r}\circ T_{q_b}$ so it is enough  to consider the cases $q\in \H$ and $q\in\R^4$.

If $q\in \H$, then  we have
\[T_q(v)=qv\overline{q}^*=qv\overline{q}=q(v_0+u\vec{v})\overline{q}=qv_0\overline{q}+qu\vec{v}\overline{q}=v_0q\overline{q}+uq\vec{v}\overline{q}=v_0+uq\vec{v}\overline{q},\] 
where we are using the fact that $v_0$ is real and $u$ is in the center of $\HU$. The fact that $q\in \H$ implies $q\vec{v}\overline{q}\in Vec(\H)$  since this is a rotation. This proofs the claim for this case.

If $q\in \R^4$ then
\begin{align*}
T_q(v)&=qv\overline{q}^*=qvq=\left(q_0+u\vec{q}\right)\left(v_0+u\vec{v}\right)\left(q_0+u\vec{q}\right)\\
&=\left(q_0+u\vec{q}\right)\left(v_0q_0+uv_0\vec{q}+uq_0\vec{v}-u^2\langle\vec{v},\vec{q}\rangle+u^2[\vec{v},\vec{q}]\right)\\
&=v_0\left(q_0^2+u^2\vec{q}^2\right)+u\vec{v}\left(q_0^2-u^2\vec{q}^2\right)-2u^2\langle\vec{v},\vec{q}\rangle\left(q_0+u\vec{q}\right)+2uv_0q_0\vec{q}\\
&=v_0\left(q_0^2+u^2\vec{q}^2\right)+u\vec{v}-2u^2\langle\vec{v},\vec{q}\rangle\left(q_0+u\vec{q}\right)+2uv_0q_0\vec{q}.
\end{align*}

Since $Q(q)=1$, we can write $q=\cos(u\theta)+\hat{q}\sin(u\theta)$ where $\hat{q}\in Vec(\H)$ and $\hat{q}^2=-1$, hence
\begin{align}
T_q(v)&=v_0\left(\cos^2(u\theta)-\sin^2(u\theta)\right)+u\vec{x}-2u\langle\vec{v},\hat{q}\rangle\cos(u\theta)\sin(u\theta)-\notag\\
&\quad-2u\langle\vec{v},\hat{q}\rangle\sin^2(u\theta)\hat{q}+2v_0\cos(u\theta)\sin(u\theta)\hat{q}\notag\\
&=v_0\cos(2u\theta)+u\vec{v}-u\langle\vec{v},\hat{q}\rangle\sin(2u\theta)-2u\langle\vec{v},\hat{q}\rangle\sin^2(u\theta)\hat{q}+v_0\sin(2u\theta)\hat{q}\label{important}
\end{align}

The claim follows from {\bf Remark 1} since in the real part sine functions are always multiplied an even number of times and in the vector part they are always multiplied an odd number of times.
\quad
 
\end{proof}

By completing the orthonormal set $\{1,\hat q\}$ to an orthonormal base $\gamma=\{1,\hat q,\hat{q}_2,\hat{q}_3\}$ of $\R^4$, and writing $v=v_0+uv_1\hat q+uv_2\hat{q}_2+uv_3\hat{q}_3$ from  Equation (\ref{important}) one obtains
\begin{align*}
T_q(v)&=v_0\cos(2u\theta)+u\left(\langle\vec{v},\hat{q}\rangle\hat{q}+\langle\vec{v},\hat{q}_2\rangle\hat{q}_2+\langle\vec{v},\hat{q}_3\rangle\hat{q}_3\right)-u\langle\vec{v},\hat{q}\rangle\sin(2u\theta)\\
&\quad-2u\langle\vec{v},\hat{q}\rangle\sin^2(u\theta)\hat{q}+v_0\sin(2u\theta)\hat{q}\\
&=v_0\cos(2u\theta)-uv_1\sin(2u\theta)+v_0\sin(2u\theta)\hat{q}+uv_1\cos(2u\theta)\hat{q}\\
&\quad+uv_2\hat{q}_2+uv_3\hat{q}_3.
\end{align*}

Therefore, we may write
\[[T_q]_\gamma=\left(\begin{array}{cccc}
\cos(2u\theta)&-\sin(2u\theta)&0&0\\
\sin(2u\theta)&\cos(2u\theta)&0&0\\
0&0&1&0\\
0&0&0&1\\
\end{array}\right)\]

\begin{proposition}\label{PropBilinearForm} The transformation $T_q$ preserves the bilinear form, i.e., if $v'=T_q(v)$ and $w'=T_q(w)$, then
	\[\left\langle v',w'\right\rangle_{\mathbb{H}}=\left\langle v,w\right\rangle_{\mathbb{H}}.\]  
\end{proposition}
\begin{proof}
As $v'=T_q(v)$ and $w'=T_q(w)$, then
\begin{align*}
\left\langle v',w'\right\rangle_{\mathbb{H}}&=\mathrm{Sc}\left(v'\overline{w'}\right)=\frac{1}{2}\left(v'\overline{w'}+w'\overline{v'}\right)\\
&=\frac{1}{2}\left(qv\overline{q}^*\overline{qw\overline{q}^*}+qw\overline{q}^*\overline{qv\overline{q}^*}\right)\\
&=\frac{1}{2}\left(qv\overline{q}^*q^*\overline{w}\,\overline{q}+qw\overline{q}^*q^*\overline{v}\,\overline{q}\right)\\
&=q\left[\frac{1}{2}\left(v\overline{w}+w\overline{v}\right)\right]\overline{q}\\
&=q\mathrm{Sc}\left(v\overline{w}\right)\overline{q}=q\overline{q}\mathrm{Sc}\left(v\overline{w}\right)=\left\langle v,w\right\rangle_{\mathbb{H}}.
\end{align*}
\end{proof}

If $v\in \mathbb{R}^4$ then, since $Q(v)$ commutes with every element of $\HU$, we have that \[Q(T_q(v))=qv\overline q^*\overline{qv\overline q^*}=qv\overline q^*q^*\overline {v}\hspace{.1cm}\overline q=qv\overline {v}\hspace{.1cm}\overline q=qQ(v)\overline{q}=Q(v)q\overline{q}=Q(v).\]
Hence, this transformation preserves the quadratic form. This implies that the space $M:=\{v\in \mathbb R^4\mbox{ }|\mbox{ }Q(v)=1\}$ is invariant by $T_q$,  so we can consider $T_q|_M$.

\section{Distance on M}\label{sec:distance}

In this section we define a distance $d$ on $M$, in general $M$ will be a space form with respect to this distance, this will be clear in section \ref{sec:spaceforms}. We will also prove that $T_q$ is an isometry of the pair $(M,d)$. 

Let $v,w\in M$, then  we can write $v=\cos(u\rho_1)+\hat{r}_1\sin(u\rho_1)$ and $w=\cos(u\rho_2)+\hat{r}_2\sin(u\rho_2)$. We define $d(v,w)$ implicitly in the following way. Since $Q(v)=Q(w)=1$ then $Q(v\overline{w})=1$, so we can write $v\overline{w}=\cos(\psi)+\hat{r}_3\sin(\psi)$ with $Q(\hat{r}_3)=1$, $\hat{r}_3^2=-1$ and $\psi\in\mathbb U$. Moreover, since

\begin{align*}
v\overline{w}&=\left(\cos(u\rho_1)+\hat{r}_1\sin(u\rho_1)\right)\left(\cos(u\rho_2)-\hat{r}_2\sin(u\rho_2)\right)\\
&=\cos(u\rho_1)\cos(u\rho_2)+\langle\hat r_1,\hat r_2 \rangle\sin(u\rho_1)\sin(u\rho_2)\\
&+ \hat{r}_1\sin(u\rho_1)\cos(u\rho_2)-\hat{r}_2\cos(u\rho_1)\sin(u\rho_2)-[\hat r_1,\hat r_2]\sin(u\rho_1)\sin(u\rho_2). 
\end{align*}
we have the following equalities
\begin{align}
\cos(\psi)&=\cos(u\rho_1)\cos(u\rho_2)+\langle\hat r_1,\hat r_2\rangle\sin(u\rho_1)\sin(u\rho_2)\label{cosines}\\
\hat{r}_3\sin(\psi)&=\hat{r}_1\sin(u\rho_1)\cos(u\rho_2)-\hat{r}_2\cos(u\rho_1)\sin(u\rho_2)-\notag\\
&\quad-[\hat r_1,\hat r_2]\sin(u\rho_1)\sin(u\rho_2).\label{sines}
\end{align}

One can prove that $\psi=u\gamma$ for some  $\gamma\in\R$, so that we define $d(v,w)=|\gamma|$.

\begin{proposition} $d:M\times M\rightarrow \R$ is a distance function.\end{proposition}
\begin{proof} We prove the tree properties of a distance function:
\begin{enumerate}
\item Obviously $d(v,w)\geq0$. If $v=w$ then $v\overline{w}=v\overline{v}=1$ and $1=cos(u0)+\hat{p}sin(u0)$ so $d(v,w)=0$.
Conversely, if $d(v,w)=0$ then this implies $v\overline{w}=cos(u0)+\hat{p}sin(u0)=1$ so $\overline{w}=v^{-1}$ and since $v\overline{v}=1$ then $\overline{v}=\overline{w}$, therefore $v=w$.
\item Suppose  $w\overline{v}=\overline{v\overline{w}}$ so if  $v\overline{w}=\cos(u\gamma)+\hat{r}\sin(u\gamma)$ then $w\overline v=\cos(u\gamma)+\overline{\hat{r}}\sin(u\gamma)$, hence $d(v,w)=d(w,v)$.

\item We now prove the triangle inequality. Let $u,v,w\in M$, suppose that $u\overline{v}= \cos(u\gamma_1)+\hat{r}_1\sin(u\gamma_1)$ and 
 $v\overline{w}= \cos(u\gamma_2)+\hat{r}_2\sin(u\gamma_2)$, then $d(u,v)+d(v,w)=|\gamma_1|+|\gamma_2|$. Observe now that
$u\overline w=u\overline{v}v\overline{w}$ since $v\in M$. Hence

\begin{align*}
u\overline w&=(\cos(u\gamma_1)+\hat{r}_1\sin(u\gamma_1))(\cos(u\gamma_2)+\hat{r}_2\sin(u\gamma_2))\\
&=\cos(u\gamma_1)\cos(u\gamma_2)-\langle\hat r_1,\hat r_2 \rangle\sin(u\gamma_1)\sin(u\gamma_2)\\
&+ \hat{r}_1\sin(u\gamma_1)\cos(u\gamma_2)+\hat{r}_2\cos(u\gamma_1)\sin(u\gamma_2)+[\hat r_1,\hat r_2]\sin(u\gamma_1)\sin(u\gamma_2). 
\end{align*}

so 

\begin{equation}\label{cosinelaw}
\cos(u\gamma_3)=\cos(u\gamma_1)\cos(u\gamma_2)-\langle\hat r_1,\hat r_2 \rangle\sin(u\gamma_1)\sin(u\gamma_2)
\end{equation}

which is a generalized law of cosines, the claim follows from Remarks \ref{cosineremark1},\ref{cosineremark2} and \ref{cosineremark3}.

\end{enumerate}

\end{proof}

If $v,w\in M$, then $T_q(v)=qv\overline{q}^*$ and $T_q(w)=qw\overline{q}^*$, hence
\[T_q(v)\overline{T_q(w)}=qv\overline{q}^*\overline{qw\overline{q}^*}=qv\overline{q}^*q^*\overline w \hspace{.1cm}\overline{q}=qv\overline w \hspace{.1cm}\overline{q}.\]

Now observe that if $v\overline{w}=\cos(u\rho)+\hat{r}\sin(u\rho)$ then $T_q(v)\overline{T_q(w)}=qv\overline w\hspace{.1cm}\overline{q}=\cos(u\rho)+q\hat{r}\overline{q}\sin(u\rho)$. Since  $(q\hat{r}\overline{q})^2=-1$  this proves the following 

\begin{theorem} Let $v,w\in M$ and $q\in\HU$ such that Q(q)=1, then $d(T_q(v),T_q(w))=d(v,w)$, i.e. $T_q$ is an isometry of $(M,d)$. 

\end{theorem}

\section{Symmetries of simply connected space forms}\label{sec:spaceforms}

In this section we restrict our results to each of the three cases. In each case express the decomposition $q=q_rq_b$, using this decomposition we describe the action of $T_q$ in $\R^4$, then we also describe the subset $M$ and the distance $d$, explaining the action of $T_q$ as an isometry of
$(M,d)$. We recover the geometry of simply connected space forms and their group of symmetries.

\subsection{Biquaternions}

Note that for $q\in\mathbb{H}(\mathbb{C})$:
\begin{align*}
Q(q)&=q\overline{q}=\overline{q}q=\sum_{k=0}^3q_k^2=\left|\alpha\right|^2-\left|\beta\right|^2+i\left(\alpha\overline{\beta}+\beta\overline{\alpha}\right)\notag\\
&=\left|\alpha\right|^2-\left|\beta\right|^2+2i\left\langle \alpha,\beta\right\rangle\in\mathbb{C}.
\end{align*}

For $z\in\mathbb{C}$, it is well-known that
\begin{align*}
\cos(z)&=\cos(a+ib)=\cos(a)\cos(ib)-\sin(a)\sin(ib)\\
&=\cos(a)\cosh(b)-i\sin(a)\sinh(b),\\
\sin(z)&=\sin(a+ib)=\sin(a)\cos(ib)+\cos(a)\sin(ib)\\
&=\sin(a)\cosh(b)+i\cos(a)\sinh(b).
\end{align*}

In particular, we have that

\[\cos(ib)=\cosh(b)\]

\[\sin(ib)=i\sinh(b)\]

which relates these functions with the hyperbolic functions.

\subsubsection{Lorentz transformations of Minkowski space}

In this case observe that if $v\in\mathbb{R}^4$  then 
\[Q(v)=Q(v_0+i\Vec{v})=(v_0+i\Vec{v})(v_0-i\Vec{v})=(v_0)^2-||\Vec{v}||^2\]
hence {\bf$\mathbb R^4$ is Minkowski spacetime}.

We proved in Theorem \ref{decomposition} that any unit-norm biquaternion $q\in\mathbb{H}(\mathbb{C})$ can be decomposed into the form
\begin{equation*}
q=q_r q_b,
\end{equation*}
where $q_r=q_r^*$ can be written in the form
\begin{equation*}
q_r=\cos\left(\frac{\theta}{2}\right)+\hat{q}_r\sin\left(\frac{\theta}{2}\right),
\end{equation*}
and $q_b^*=\overline{q_b}$ can be written in the form
\begin{align*}
q_b&=\cos\left(i\frac{\phi}{2}\right)+\hat{q}_b\sin\left(i\frac{\phi}{2}\right)\\
&=\cosh\left(\frac{\phi}{2}\right)+i\hat{q}_b\sinh\left(\frac{\phi}{2}\right).
\end{align*}

If $v\in\mathbb{R}^4$ represents a space-time event then a proper
Lorentz transformation can be written in terms of a biquaternion
$q\in\mathbb{H}(\mathbb{C})$ as
\[T_q(v):=qv\overline{q}^*\quad\text{with}\quad Q(1)=1.\]
The constraint
\[Q(q)=1,\]
that implies the following two conditions:
\begin{align*}
\left|\alpha\right|^2-\left|\beta\right|^2&=1,\\
\left\langle \alpha,\beta\right\rangle&=0,
\end{align*}
ensures that this transformation is 6-dimensional. Besides, Proposition \ref{PropBilinearForm} shows that this transformation preserves the value of the bilinear form and so must be a Lorentz transformation.

Then, by equation (\ref{morphism}) we have that
\begin{equation}
T_q=T_{q_r q_b}=T_{q_r}\circ T_{q_b} .
\end{equation}

As $q_r=q_r^*$, then $q_r\in\mathbb{H}$ and the transformation $T_{q_r}$ describes a spatial rotation. On the other side, observe that
\begin{equation*}
T_{q_b}(v)=v_0\cosh(\phi)+\langle\vec{v},\hat{q}_b\rangle\sinh(\phi)+i\vec{v}+2i\langle\vec{v},\hat{q}_b\rangle\sinh^2\left(\frac{\phi}{2}\right)\hat{q}_b+iv_0\sinh(\phi)\hat{q}_b.
\end{equation*}

If $v'=T_{q_b}(v)=v_0'+i\vec{v}\,'$, then
\begin{align*}
v_0'&=v_0\cosh(\phi)+\langle\vec{v},\hat{q}_b\rangle\sinh(\phi)\\
\vec{v}\,'&=\vec{v}+2\langle\vec{v},\hat{q}_b\rangle\sinh^2\left(\frac{\phi}{2}\right)\hat{q}_b+v_0\sinh(\phi)\hat{q}_b,
\end{align*}
which is the standard form of a pure Lorentz transformation (a ``boost'').

Furthermore, the corresponding matrix takes the form:
\[[T_{q_b}]_\gamma=\left(\begin{array}{cccc}
\cosh(\phi)&-i\sinh(\phi)&0&0\\
i\sinh(\phi)&\cosh(\phi)&0&0\\
0&0&1&0\\
0&0&0&1\\
\end{array}\right)\]

Hence, the decomposition $q=q_rq_b$ corresponds to the decompostion of a Lorentz transformation $T_q$ as a spatial rotation $T_{q_r}$ and a Lorentz boost $T_{q_b}$.
\subsubsection{Orientation preserving isometries of $H^3$}

We now restrict our study to $M$. If $v\in M$ then $v=v_0+i\Vec{v}$ and $Q(v)=1$, i.e.

\[1=(v_0+i\Vec{v})(v_0-i\Vec{v})=(v_0)^2-||\Vec{v}||^2\] 

hence is the two sheeted hyperboloid of revolution, if we restrict to a sheet we can identify $M$ with $H^3$ the hyperbolic space.
For $v,w\in H^3$, such that $v=\cosh(\rho_1)+i\hat{r}_1\sinh(\rho_1)$ and $w=\cosh(u\rho_2)+i\hat{r}_2\sinh(\rho_2)$ we have that
\begin{align*}
v\overline{w}&=\cosh(\rho_1)\cosh(\rho_2)-\langle\hat r_1,\hat r_2\rangle\sinh(\rho_1)\sinh(\rho_2)\\
&+i\hat{r}_1\sinh(\rho_1)\cosh(\rho_2)-i\hat{r}_2\cosh(\rho_1)\sinh(\rho_2)-[\hat r_1,\hat r_2]\sinh(\rho_1)\sinh(\rho_2),\notag\\
\end{align*}

Hence, from the definition of distance given in section \ref{sec:distance},\[d(v,w)=\text{arccosh}[\cosh(\rho_1)\cosh(\rho_2)-\langle\hat r_1,\hat r_2\rangle\sinh(\rho_1)\sinh(\rho_2)]\]
which is the hyperbolic distance on $H^3$.

\begin{remark}\label{cosineremark1} In this case Equation (\ref{cosinelaw}) becomes the well known law of cosines for hyperbolic triangles.
\end{remark}
Moreover, for every $q$ such that $Q(q)=1$,  $T_q$ is an isometry of $(H^3,d_{hyperbolic})$, so we have recover the group of orientation preserving isometries. If we write $T_q=T_{q_r}\circ T_{q_b}$, then $T_{q_r}$ is a rotation and  $T_{q_b}$ is a hyperbolic translation. The hole picture is summarized in the following illustrative figure.

\begin{figure}[h]
\centering

\begin{tikzpicture}[scale=.6]

\draw [thick,->]  (0,-6) -- (0,6);
\draw [thick] (-3,2) -- (6,2) -- (3,-2) -- (-6,-2) -- (-3,2);

\draw [thick] (0,0) -- (.5,2.05);
\draw [thick] (0,0) -- (2.3,3.05);

 \draw [thick,->] (-1.4,-4) arc(200:340:1.5cm and .8cm);

\draw (0,4.9) ellipse (4.5cm and 1.5cm);
\draw [thick] plot[domain=-2.5:2.5] ({sinh(\x)},{cosh(\x)});

\draw [thick,->] plot[domain=0.25:1] ({2*sinh(\x)},{2*cosh(\x)});

\node at (.5,-4.2)  {$\theta$};
\draw [thick] (-1.4,-4) -- (0,-3.5) -- (1.4,-4);
\node at (1.2,2.7)  {$\phi$};
\node at (5,1.5)  {$\mathbb{R}^3$};
\node at (-5,3)  {$\mathbb{R}^4$};
 
\node at (-2,3)  {$H^3$};
\end{tikzpicture}
\caption{Action of $q=q_rq_b$, where $q_r=\cos(\theta/2)+\sin(\theta/2)\hat{q}_r$ and $q_b=\cos(\phi/2)+i\sinh(\phi/2)\hat{q}_b$.}

\end{figure}
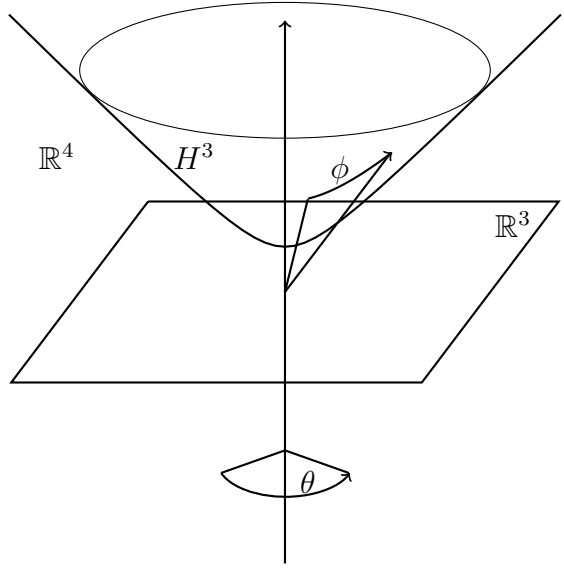

\subsection{Dual quaternions}

For $q\in\mathbb{H}(\mathbb{D})$:
\begin{align*}
q\overline{q}&=\overline{q}q=\sum_{k=0}^3q_k^2=\left|\alpha\right|^2+\epsilon\left(\alpha\overline{\beta}+\beta\overline{\alpha}\right)\notag\\
&=\left|\alpha\right|^2+2\epsilon\left\langle \alpha,\beta\right\rangle\in\mathbb{D},
\end{align*}

For $z\in\mathbb{D}$, using Taylor's series one obtains
\begin{align*}
\cos(z)&=\cos(a+\epsilon b)=\cos(a)\cos(\epsilon b)-\sin(a)\sin(\epsilon b)\\
&=\cos(a)-\epsilon\sin(a)b,\\
\sin(z)&=\sin(a+\epsilon b)=\sin(a)\cos(\epsilon b)+\cos(a)\sin(\epsilon b)\\
&=\sin(a)+\epsilon\cos(a)b.
\end{align*}

In particular, we have that

\[\cos(\varepsilon b)=1\]
\[\sin(\varepsilon b)=\varepsilon b\]

\subsubsection{Galilean four dimensional space}

In this case observe that if $v\in\mathbb{R}^4$  then 
\[Q(v)=Q(v_0+\varepsilon\Vec{v})=(v_0+\varepsilon \Vec{v})(v_0-\varepsilon\Vec{v})=(v_0)^2\]
hence $\mathbb R^4$ has a degenerate metric which is called the {\bf Galilean space}.

We proved in Theorem \ref{decomposition} that any unit-norm dual quaternion $q\in\mathbb{H}(\mathbb{D})$ can be decomposed into the form
\begin{equation*}
	q=q_r q_b,
\end{equation*}
where $q_r=q_r^*$ can be written in the form
\begin{equation*}
	q_r=\cos\left(\frac{\theta}{2}\right)+\hat{q}_r\sin\left(\frac{\theta}{2}\right),
\end{equation*}
and $q_b^*=\overline{q_b}$ can be written in the form
\begin{align*}
	q_b&=\cos\left(\varepsilon\frac{\phi}{2}\right)+\hat{q}_b\sin\left(\varepsilon\frac{\phi}{2}\right)\\
	&=1+\varepsilon\hat{q}_b\frac{\phi}{2}.
\end{align*}

If $v\in\mathbb{R}^4$, then a  Galilean symmetry  can be written in terms of a dual quaternion
$q\in\mathbb{H}(\mathbb{D})$ as
\[T_q(v):=qv\overline{q}^*\quad\text{with}\quad Q(1)=1.\]
The constraint
\[Q(q)=1,\]
implies the following two conditions:
\begin{align*}
\left|\alpha\right|^2&=1,\\
\left\langle \alpha,\beta\right\rangle&=0.
\end{align*}

Besides, Proposition \ref{PropBilinearForm} shows that this transformation preserves the value of the bilinear form.

Then, by equation (\ref{morphism}) we have that
\begin{equation}
T_q=T_{q_r q_b}=T_{q_r}\circ T_{q_b} .
\end{equation}

As $q_r=q_r^*$, then $q_r\in\mathbb{H}$ and the transformation $T_{q_r}$ describes a spatial rotation. On the other side, observe that
\begin{equation}
	T_{q_b}(v)=v_0+\varepsilon\left(\vec{v}+ v_0\phi\hat{q}_b\right).
\end{equation}

If $v'=T_{q_b}(v)=v_0'+\varepsilon\vec{v}\,'$, then
\begin{align*}
v_0'&=v_0\\
\vec{v}\,'&=\vec{v}+ v_0\phi\hat{q}_b,
\end{align*}
which represents a shear.

Furthermore, the corresponding matrix takes the form:

\[[T_{q_b}]_\gamma=\left(\begin{array}{cccc}
1&-\varepsilon\phi&0&0\\
\varepsilon\phi&1&0&0\\
0&0&1&0\\
0&0&0&1\\
\end{array}\right)\]

Hence, the decomposition $q=q_rq_b$ corresponds to the decompostion of a Galilean symmetry  $T_q$ as a spatial rotation $T_{q_r}$ and a shear $T_{q_b}$.

\subsubsection{Orientation preserving isometries of $E^3$}

We now restrict our study to $M$. If $v\in M$ then $v=v_0+\varepsilon\Vec{v}$ and $Q(v)=1$, i.e.

\[1=(v_0+\varepsilon\Vec{v})(v_0-\varepsilon\Vec{v})=(v_0)^2\] 
hence $M$ is the union of the two linear spaces defined by equations $v_0=\pm1$, if we restrict to the plane $v_0=1$ we can identify $M$ with $E^3$ the Euclidean space.

For $v,w\in M$, such that $v=1+\varepsilon\hat{r}_1\rho_1$ and $w=1+\varepsilon\hat{r}_2\rho_2$, we have that
\begin{align*}
v\overline{w}=&1+\hat{r}_1\varepsilon\rho_1-\hat{r}_2\varepsilon\rho_2=1+\varepsilon(\hat{r}_1\rho_1-\hat{r}_2\rho_2).
\end{align*}
Observe that if $||\hat{r}_1\rho_1-\hat{r}_2\rho_2||\neq 0$ 
\[\hat{r}_1\rho_1-\hat{r}_2\rho_2={||\hat{r}_1\rho_1-\hat{r}_2\rho_2||}\frac{\hat{r}_1\rho_1-\hat{r}_2\rho_2}{||\hat{r}_1\rho_1-\hat{r}_2\rho_2||}\]
and that
\[\left(\frac{\hat{r}_1\rho_1-\hat{r}_2\rho_2}{||\hat{r}_1\rho_1-\hat{r}_2\rho_2||}\right)^2=-1\]
therefore,  from the definition of distance given in section \ref{sec:distance},

 \[d(v,w)=||\hat{r}_1\rho_1-\hat{r}_2\rho_2||\]

which is the Euclidean distance on $E^3$.

\begin{remark} \label{cosineremark2}Observe that if $u\overline{v}=1+\varepsilon\hat{r_1}\rho_1$ and $v\overline{w}=1+\varepsilon\hat{r_2}\rho_2$, so that $d(u,v)=|\rho_1|$ and $d(u,v)=|\rho_2|$, then 

\begin{eqnarray*}
u\overline{w}&=&1+\varepsilon(\hat{r}_1\rho_1+\hat{r}_2\rho_2)\\
	&=&1+\varepsilon{||\hat{r}_1\rho_1+\hat{r}_2\rho_2||}\frac{\hat{r}_1\rho_1-\hat{r}_2\rho_2}{||\hat{r}_1\rho_1+\hat{r}_2\rho_2||}
\end{eqnarray*}
hence, \[d(u,w)=||\hat{r}_1\rho_1+\hat{r}_2\rho_2||\leq|\rho_1|+|\rho_2|=d(u,v)+d(u,w)\]
\end{remark}
Moreover, for every $q$ such that $Q(q)=1$,  $T_q$ is an isometry of $(E^3,d_{euclidean})$, so we have recover the group of orientation preserving isometries. If we write $T_q=T_{q_r}\circ T_{q_b}$, then $T_{q_r}$ is a rotation and  $T_{q_b}$ is a  translation. The hole picture is summarized in the following illustrative figure.

\begin{figure}[h]
\centering

\begin{tikzpicture}[scale=.6]

\draw [thick,->]  (0,-6) -- (0,6);
\draw [thick] (-3,2) -- (6,2) -- (3,-2) -- (-6,-2) -- (-3,2);

\draw [thick] (0,0) -- (.5,3);
\draw [thick] (0,0) -- (1.5,1.5);
\draw [thick,->] (.5,3) -- (1.5,1.5);

 \draw [thick,->] (-1.4,-4) arc(200:340:1.5cm and .8cm);

\draw [thick] (-3,4) -- (6,4) -- (3,-0) -- (-6,-0) -- (-3,4);

\node at (.5,-4.2)  {$\theta$};
\draw [thick] (-1.4,-4) -- (0,-3.5) -- (1.4,-4);
\node at (1.2,2.7)  {$\phi$};
\node at (5,1.5)  {$\mathbb{R}^3$};
\node at (-5,3)  {$\mathbb{R}^4$};
 
\node at (-2,3)  {$E^3$};
\end{tikzpicture}
\caption{Action of $q=q_rq_b$, where $q_r=\cos(\theta/2)+\sin(\theta/2)\hat{q}_r$ and $q_r=1+\varepsilon\hat{q}_b\frac{\phi}{2}$.}

\end{figure}
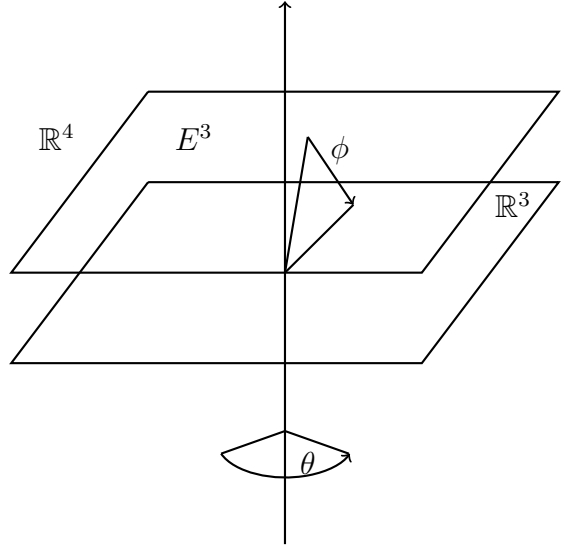

\subsection{Split biquaternions}

For $q\in\mathbb{H}(\mathbb{S})$:
\begin{align*}
q\overline{q}&=\overline{q}q=\sum_{k=0}^3q_k^2=\left|\alpha\right|^2+\left|\beta\right|^2+j\left(\alpha\overline{\beta}+\beta\overline{\alpha}\right)\notag\\
&=\left|\alpha\right|^2+\left|\beta\right|^2+2j\left\langle \alpha,\beta\right\rangle\in\mathbb{S}.
\end{align*}

For $z\in\mathbb{S}$, using Taylor's series one obtains
\begin{align*}
\cos(z)&=\cos(a+jb)=\cos(a)\cos(jb)-\sin(a)\sin(jb)\\
&=\cos(a)\cos(b)-j\sin(a)\sin(b),\\
\sin(z)&=\sin(a+jb)=\sin(a)\cos(jb)+\cos(a)\sin(jb)\\
&=\sin(a)\cos(b)+j\cos(a)\sin(b).
\end{align*}

In particular, we have that

\[\cos(jb)=\cos(b)\]
\[\sin(jb)=j \sin(b)\]

\subsubsection{ Rotations in Euclidean four dimensional space}

In this case observe that if $v\in\mathbb{R}^4$  then 
\[Q(v)=Q(v_0+j\Vec{v})=(v_0+j\Vec{v})(v_0-j\Vec{v})=(v_0)^2+||\Vec{v}||^2\]
hence {\bf$\mathbb R^4$ is the four dimensional Euclidean space}.

We proved in Theorem \ref{decomposition} that any unit-norm split biquaternion $q\in\mathbb{H}(\mathbb{S})$ can be decomposed into the form
\begin{equation*}
q=q_r q_b,
\end{equation*}
where $q_r=q_r^*$ can be written in the form
\begin{equation*}
q_r=\cos\left(\frac{\theta}{2}\right)+\hat{q}_r\sin\left(\frac{\theta}{2}\right),
\end{equation*}
and $q_b^*=\overline{q_b}$ can be written in the form
\begin{align*}
q_b&=\cos\left(j\frac{\phi}{2}\right)+\hat{q}_b\sin\left(j\frac{\phi}{2}\right)\\
&=\cos\left(\frac{\phi}{2}\right)+j\hat{q}_b\sin\left(\frac{\phi}{2}\right).
\end{align*}

If $v\in\mathbb{R}^4$ a rotation can be written in terms of a split biquaternion
$q\in\mathbb{H}(\mathbb{S})$ as
\[T_q(v):=qv\overline{q}^*\quad\text{with}\quad Q(1)=1.\]
The constraint
\[Q(q)=1,\]
implies the following two conditions:
\begin{align*}
	\left|\alpha\right|^2+	\left|\beta\right|^2&=1,\\
	\left\langle \alpha,\beta\right\rangle&=0.
\end{align*}

Besides, Proposition \ref{PropBilinearForm} shows that this transformation preserves the value of the bilinear form.

Then, by equation (\ref{morphism}) we have that
\begin{equation}
	T_q=T_{q_r q_b}=T_{q_r}\circ T_{q_b} .
\end{equation}

As $q_r=q_r^*$, then $q_r\in\mathbb{H}$ and the transformation $T_{q_r}$ describes a simple rotation. On the other side, observe that
\begin{equation}
T_{q_b}(v)=v_0\cos(\phi)-\langle\vec{v},\hat{q}_b\rangle\sin(\phi)+j\vec{v}-2j\langle\vec{v},\hat{q}_b\rangle\sin^2\left(\frac{\phi}{2}\right)\hat{q}_b+jv_0\sin(\phi)\hat{q}_b.
\end{equation}
If $v'=T_{q_b}(v)=v_0'+j\vec{v}\,'$, then
\begin{align*}
	v_0'&=v_0\cos(\phi)-\langle\vec{v},\hat{q}_b\rangle\sin(\phi),\\
	\vec{v}\,'&=\vec{v}-2\langle\vec{v},\hat{q}_b\rangle\sin^2\left(\frac{\phi}{2}\right)\hat{q}_b+v_0\sin(\phi)\hat{q}_b,
\end{align*}
which represents another simple rotation.

Furthermore, the corresponding matrix takes the form:

\[[T_{q_b}]_\gamma=\left(\begin{array}{cccc}
\cos(\phi)&-j\sin(\phi)&0&0\\
j\sin(\phi)&\cos(\phi)&0&0\\
0&0&1&0\\
0&0&0&1\\
\end{array}\right)\]

Hence, the decomposition $q=q_rq_b$ corresponds to the decompostion of a four dimensional rotation $T_q$ as two simple rotations rotation $T_{q_r}$ and  $T_{q_b}$, moreover this is a construction of a double cover $S^3\times S^3\rightarrow SO(4)$.

\subsubsection{Orientation preserving isometries of $S^3$}

We now restrict our study to $M$. If $v\in M$ then $v=v_0+i\Vec{v}$ and $Q(v)=1$, i.e.

\[1=(v_0+j\Vec{v})(v_0-j\Vec{v})=(v_0)^2+||\Vec{v}||^2\] 

hence, $M$ is $S^3$ the sphere of dimension 3.

For $v,w\in M$, such that $v=\cos(\rho_1)+j\hat{r}_1\sin(\rho_1)$ and $w=\cos(u\rho_2)+j\hat{r}_2\sin(\rho_2)$ we have that
\begin{align*}
	v\overline w&=\cos(\rho_1)\cos(\rho_2)+\langle\hat r_1,\hat r_2\rangle\sin(\rho_1)\sin(\rho_2)\\
&+j\hat{r}_1\sin(\rho_1)\cos(\rho_2)-j\hat{r}_2\cos(\rho_1)\sin(\rho_2)-\notag\\
	&\quad-[\hat r_1,\hat r_2]\sin(\rho_1)\sin(\rho_2).
\end{align*}

Hence, from the definition of distance given in section \ref{sec:distance}\[d(v,w)=\text{arccos}[\cos(\rho_1)\cos(\rho_2)+\langle\hat r_1,\hat r_2\rangle\sin(\rho_1)\sin(\rho_2)]\]
which is the spherical distance on $S^3$.

\begin{remark}\label{cosineremark3} In this case Equation (\ref{cosinelaw}) becomes the well known law of cosines for spherical  triangles.
\end{remark}
Moreover, for every $q$ such that $Q(q)=1$,  $T_q$ is an isometry of $(S^3,d_{round})$, so we have recover the group of orientation preserving isometries. If we write $T_q=T_{q_r}\circ T_{q_b}$, then $T_{q_r}$  and  $T_{q_b}$ are simple rotations. The hole picture is summarized in the following illustrative figure.

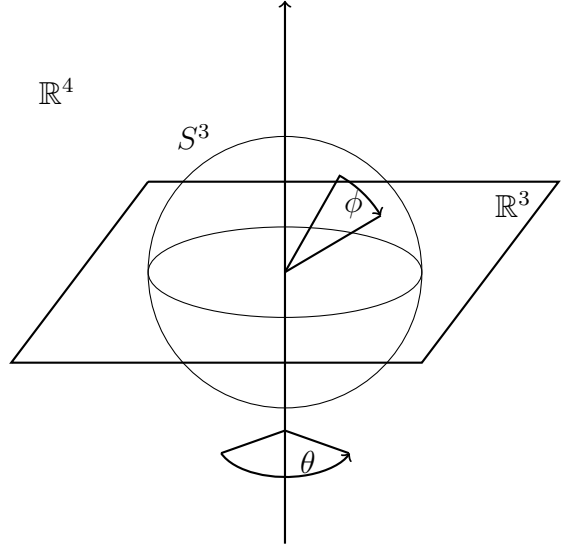
\begin{figure}[h]
\centering

\begin{tikzpicture}[scale=.6]

\node at (-5,4)  {$\mathbb{R}^4$};
\draw [thick,->]  (0,-6) -- (0,6);
\draw [thick] (-3,2) -- (6,2) -- (3,-2) -- (-6,-2) -- (-3,2);

\draw [thick] (0,0) -- (2.1,1.25);
\draw [thick] (0,0) -- (1.21,2.15);
\draw (0,0) ellipse (3cm and 1cm);
\draw (0,0) circle(3);
\draw [thick,<-] (2.1,1.25)  arc [radius=2.5, start angle=31, end angle=60];
 \draw [thick,->] (-1.4,-4) arc(200:340:1.5cm and .8cm);

\node at (.5,-4.2)  {$\theta$};
\draw [thick] (-1.4,-4) -- (0,-3.5) -- (1.4,-4);
\node at (1.5,1.5)  {$\phi$};
\node at (5,1.5)  {$\mathbb{R}^3$};
 
\node at (-2,3)  {$S^3$};
\end{tikzpicture}
\caption{Action of $q=q_rq_b$, where $q_r=\cos(\theta/2)+\sin(\theta/2)\hat{q}_r$ and $q_b=\cos(\phi/2)+j\sin(\phi/2)\hat{q}_b$.}

\end{figure}

\section*{Acknowledgments}
 The first author acknowledges partial support from CONACyT under grant 256126.

\end{document}